\documentclass[a4, 12pt, 
]{amsart}
\usepackage[utf8]{inputenc}
\usepackage[T1]{fontenc}

\usepackage[colorlinks=true, linkcolor=blue, urlcolor=magenda, citecolor=red]{hyperref}

\usepackage{color}
\usepackage{xcolor} 
\usepackage{ulem}
\usepackage{calligra}
\usepackage{mathrsfs}

\usepackage[all]{xy}

\newtheorem{theo}{Theorem}[section]

\newtheorem{lemm}[theo]{Lemma}
\newtheorem{cor}[theo]{Corollary}

\numberwithin{equation}{section}

\theoremstyle{definition}
\newtheorem{defi}[theo]{Definition}

\theoremstyle{remark}
\newtheorem{rem}[theo]{Remark}

\newcommand{\Alb}[0]{\operatorname{Alb}}

\newcommand{\Image}[0]{\operatorname{Im}\hspace{-0.03cm}}

\newcommand{\End}[0]{\operatorname{End}}
\newcommand{\rank}[0]{\operatorname{rank}}
\newcommand{\codim}[0]{\operatorname{codim}}

\newcommand{\GL}[0]{\operatorname{GL}}
\newcommand{\Aut}[0]{\operatorname{Aut}}
\newcommand{\Unv}[1]{{#1}_{\rm{univ}}}

\newcommand{\reg}{{\rm{reg}}}
\newcommand{\sing}{{\rm{sing}}}
\newcommand{\sat}{{\rm{sat}}}

\newcommand{\TX}[2]{\mathrm{T}_{#1#2}}

\makeatletter
\newcommand*{\da@rightarrow}{\mathchar"0\hexnumber@\symAMSa 4B }
\newcommand*{\da@leftarrow}{\mathchar"0\hexnumber@\symAMSa 4C }
\newcommand*{\xdashrightarrow}[2][]{%
  \mathrel{%
    \mathpalette{\da@xarrow{#1}{#2}{}\da@rightarrow{\,}{}}{}%
  }%
}
\newcommand{\xdashleftarrow}[2][]{%
  \mathrel{%
    \mathpalette{\da@xarrow{#1}{#2}\da@leftarrow{}{}{\,}}{}%
  }%
}
\newcommand*{\da@xarrow}[7]{%
  \sbox0{$\ifx#7\scriptstyle\scriptscriptstyle\else\scriptstyle\fi#5#1#6\m@th$}%
  \sbox2{$\ifx#7\scriptstyle\scriptscriptstyle\else\scriptstyle\fi#5#2#6\m@th$}%
  \sbox4{$#7\dabar@\m@th$}%
  \dimen@=\wd0 %
  \ifdim\wd2 >\dimen@
    \dimen@=\wd2 %
  \fi
  \count@=2 %
  \def\da@bars{\dabar@\dabar@}%
  \@whiledim\count@\wd4<\dimen@\do{%
    \advance\count@\@ne
    \expandafter\def\expandafter\da@bars\expandafter{%
      \da@bars
      \dabar@ 
    }%
  }%
  \mathrel{#3}%
  \mathrel{%
    \mathop{\da@bars}\limits
    \ifx\\#1\\%
    \else
      _{\copy0}%
    \fi
    \ifx\\#2\\%
    \else
      ^{\copy2}%
    \fi
  }%
  \mathrel{#4}%
}
\makeatother

\usepackage{geometry}
\begin{document}

\title[Compact K\"ahler manifolds with semi-positive HSC]
{Fundamental groups of \\ compact K\"ahler manifolds with \\ semi-positive holomorphic sectional curvature}

\author{Shin-ichi MATSUMURA}

\address{
Mathematical Institute 
$\&$ Division for the Establishment of Frontier Science of Organization for Advanced Studies, 
Tohoku University, 
6-3, Aramaki Aza-Aoba, Aoba-ku, Sendai 980-8578, Japan.}

\email{{\tt mshinichi-math@tohoku.ac.jp}}
\email{{\tt mshinichi0@gmail.com}}

\date{\today, version 0.01}

\renewcommand{\subjclassname}{%
\textup{2020} Mathematics Subject Classification}


\maketitle

\begin{abstract}
In this paper, we prove that a compact K\"ahler manifold $X$ with semi-positive holomorphic sectional curvature admits a locally trivial fibration $\phi \colon X \to Y$, 
where the fiber $F$ is a rationally connected projective manifold  
and the base $Y$ is a finite \'etale quotient of a torus. 
This result extends the structure theorem, previously established for projective manifolds, to compact K\"ahler manifolds. 
A key part of the proof involves analyzing the foliation generated by truly flat tangent vectors and showing the abelianness of the topological fundamental group $\pi_{1}(X)$, with a focus on varieties of special type.
\end{abstract}

\tableofcontents

\section{Introduction}\label{Sec-1}

One of the central problems in geometry is decomposing varieties into basic building blocks to understand their geometric properties better. Building on foundational works \cite{Mor79, SY80}, structure theorems were established in the 1980s and 1990s for varieties with \lq\lq semi-positive'' curvature, including compact K\"ahler manifolds with semi-positive holomorphic bisectional curvature \cite{HSW81, MZ86, Mok88} and varieties with nef tangent bundle \cite{CP91, DPS94},  decomposing them into components with \lq\lq positive'' and \lq\lq flat'' curvature. More recently, these theorems have been extended to semi-positive curvature conditions, such as nef anti-canonical bundles \cite{Cao19, CH19, DB22, Zha05}, 
and pseudo-effective tangent bundles \cite{HIM22, IMZ, Mul}, 
semi-positive holomorphic sectional curvatures \cite{Yan18, HW20, Mat22}. 
While significant progress has been made for projective varieties, 
a complete theory for compact K\"ahler manifolds is still lacking.

This paper focuses on the notion of \textit{semi-positive holomorphic sectional curvatures} and completes a comprehensive structure theorem for compact K\"ahler manifolds (see \cite[Section 2]{Mat22} and references therein for the basic properties of holomorphic sectional curvatures). The main result (see Theorem \ref{thm-main} below) resolves \cite[Problem 3.7]{Mat22b} 
and a question discussed in \cite{Ni19}, and extends the structure theorem, previously established 
for projective manifolds, to compact K\"ahler manifolds.

\begin{theo}[{Main Result}]\label{thm-main}
Let $X$ be a compact K\"ahler manifold admitting a K\"ahler metric $g$ with semi-positive holomorphic sectional curvature. 
Then, the following statements hold$:$
\begin{itemize}
\item[$\rm{(a)}$] The $($topological$)$ fundamental group $\pi_{1}(X)$ of $X$ is a virtually abelian group 
$($i.e.,\,there exists an abelian subgroup of  $\pi_{1}(X)$ of finite index$).$
\item[$\rm{(b)}$] 
There exists a fibration $\phi\colon X \to Y$ $($i.e.,\,a proper surjective morphism with connected fibers$)$ 
to a compact K\"ahler manifold $Y$ 
with the following properties$:$
\begin{itemize}
\item[$\rm{(b_{1})}$] The fibration $\phi\colon X \to Y$ is a locally constant fibration 
$($see \cite[Definition 2.3]{MW} for the definition$).$ In particular, 
the fibration $\phi\colon X \to Y$ is a locally trivial fibration.

\item[$\rm{(b_{2})}$] The base $Y$ is a finite \'etale quotient of a torus 
$($i.e.,\,there exists an unramified finite surjective morphism $T \to Y$ from a compact complex torus $T$.$)$

\item[$\rm{(b_{3})}$] The fiber $F$ is a rationally connected projective manifold 
$($i.e.,\,any two points can be connected by a rational curve$)$. 
In particular, the fibration $\phi\colon X \to Y$ is an MRC $($maximal rationally connected$)$ fibration. 
\end{itemize}
\end{itemize}
\end{theo}

\begin{rem}\label{rem-property}
Under the same assumption as in Theorem \ref{thm-main}, 
the compact K\"ahler manifold $(X, g)$ also satisfies the following property:
\begin{itemize}
\item[$\rm{(b_{4})}$] There exist a K\"ahler metric $g_F$ on the fiber $F$ and 
a K\"ahler metric $g_Y$ on the base $Y$ with the following properties$:$
\begin{itemize}
\item[$-$] The holomorphic sectional curvature of $g_F$ is semi-positive.  
\item[$-$] The curvature tensor of $g_Y$ is flat. 
\item[$-$] 
Consider the splitting of the universal cover $X_{\rm{univ}}$ of $X$ obtained from $\phi$ being a locally constant fibration$:$
$$
X_{\rm{univ}} \cong \mathbb{C}^m \times F. 
$$
Then, this splitting is not only biholomorphic but also isometric 
with respect to the K\"ahler metrics $\mu^*g $, $\pi^*g_{Y}$, and $g_F$, 
where $\pi$ and $\mu$ respectively denote the universal cover $\pi \colon \mathbb{C}^m \to Y$ of $Y$ 
and the universal cover $\mu \colon  X_{\rm univ} \to X$ of $X$. 
\end{itemize}
\end{itemize}
Note that property \( (\mathrm{b_4}) \) is stated separately from Theorem \ref{thm-main} for convenience of citation in this paper, not due to any mathematical necessity.
\end{rem}

We describe the contribution of this paper to the problem called \textit{Yau's conjecture} \cite[Problem 72]{Yau82}, one of the main motivations for studying holomorphic sectional curvatures. Yau's conjecture predicts that any compact K\"ahler manifold with positive holomorphic sectional curvature is rationally connected (equivalently, the base $Y$ in Theorem \ref{thm-main} is a single point). This conjecture was fully resolved in \cite{Yan18}, following its resolution for projective manifolds in \cite{HW20}, including the quasi-positive case. The conjecture was further extended to projective manifolds with semi-positive holomorphic sectional curvature (see \cite[Corollary 1.3]{Mat22}) using the notion of truly flat tangent vectors (see \cite{HLWZ18} for the definition). As an application of Theorem \ref{thm-main}, we extend \cite[Corollary 1.3]{Mat22} to compact K\"ahler manifolds (see Corollary \ref{cor-Yau} below). The invariant $n_{{\rm{tf}}}{(X, g)}$ in the corollary can be viewed as an analog of the numerical Kodaira dimension (see \cite{HLWZ18} for the definition of the invariant $n_{{\rm{tf}}}{(X, g)}$).
Therefore, the formulation of Corollary \ref{cor-Yau} can also be seen as an analog of Hacon-M\textsuperscript{c}Kernan's question \cite[Question 3.1]{HM07}, which was affirmatively resolved in \cite{EG19, CCM21, EIM20}. This question remains open for compact K\"ahler manifolds, making Corollary \ref{cor-Yau} particularly interesting. Note that the latter part of Corollary \ref{cor-Yau} was already established in \cite{ZZ} using a different method.

\begin{cor}\label{cor-Yau}
Let $(X, g)$ be a compact K\"ahler manifold  
equipped with a K\"ahler metric $g$ with semi-positive holomorphic sectional curvature. 
Let  $\phi \colon  X \dashrightarrow Y$ be an MRC fibration of $X$. 
Then, we have  
$$
\dim X -\dim Y \geq n_{{\rm{tf}}}{(X, g)}. 
$$
In particular, the manifold $X$ is rationally connected 
if $n_{{\rm{tf}}}{(X, g)}=\dim X$ $($which is satisfied when the holomorphic sectional curvature is quasi-positive$)$. 
\end{cor}

We discuss the challenges in proving Theorem \ref{thm-main} compared to previous studies.
Conclusion (a) directly follows from Conclusion (b), and 
Conclusion (b) was proved in the case where $X$ is a projective manifold (see \cite[Theorem 1.3]{Mat22}).
In the proof of \cite[Theorem 1.3]{Mat22}, it is shown that an MRC fibration $X \dashrightarrow Y$ 
can be chosen to satisfy the desired properties in Theorem \ref{thm-main}.
The smooth base $Y$ of an MRC fibration is non-uniruled by \cite{GHS03},  
and thus the canonical bundle $K_{Y}$ is pseudo-effective by \cite{BDPP13} 
when $X$ (and consequently $Y$) is projective.
The pseudo-effectivity of $K_{Y}$ is a key point in the proof, 
but this property remains an open problem when $X$ is a compact K\"ahler manifold, 
which presents a primary difficulty in proving Conclusion (b).

Conclusion (b) might be deduced from Conclusion (a). 
However, this approach is not straightforward, and proving Conclusion (a) directly is also challenging. 
Similar results have been established for compact K\"ahler manifolds with nef anti-canonical bundles (see \cite[Th\'eor\`eme 2]{Pau97}, \cite[Theorem 1.4]{Pau17}, \cite[Theorem 2.2]{Cam95}), but adapting these methods to the context of holomorphic sectional curvatures is highly non-trivial. Even for projective manifolds, the proof of Conclusion (a) relies on Conclusion (b).

We outline the proof of Theorem \ref{thm-main}, keeping the aforementioned challenges in mind. Let $(X, g)$ be a compact K\"ahler manifold equipped with a K\"ahler metric $g$ with semi-positive holomorphic sectional curvature. The basic strategy is to consider the Albanese map $\alpha \colon X \to Y := \Alb(X)$ instead of MRC fibrations. As mentioned earlier, since $K_Y$ is a trivial line bundle (and thus pseudo-effective), the Albanese map satisfies all properties in Conclusion (b) except for property $\rm{(b_{3})}$. Therefore, our goal is to prove that the fiber $F$ of the Albanese map $\alpha \colon X \to Y$ is a rationally connected projective manifold.

If $X$ is rationally connected, the augmented irregularity $\hat{q}(X)$ is clearly zero. Conversely, if Theorem \ref{thm-main} holds, we can deduce that when the augmented irregularity $\hat{q}(X)$ vanishes, the manifold $X$ is a rationally connected projective manifold. Based on this observation, we divide the proof into two steps:
\begin{itemize}
    \item[$(1)$] Show that the fiber $F$ has vanishing augmented irregularity: $\hat{q}(F) = 0$.
    \item[$(2)$] Show that $X$ with $\hat{q}(X) = 0$ is a rationally connected projective manifold.
\end{itemize}
The desired property $\rm{(b_{3})}$ can be obtained 
by applying the result of Step (2) to the fiber $F$ that satisfies 
the conclusion $\hat{q}(F) = 0$ of Step (1).

Step (1) will be discussed in Section \ref{Sec-3}. To achieve Step (1), we assume that $X$ and $F$ attain their augmented irregularities (i.e., $q(X) = \hat{q}(X)$ and $q(F) = \hat{q}(F)$) by replacing $X$ with an appropriate finite \'etale cover. We then consider the relative Albanese map $\beta \colon X \to Z := \Alb(X/Y)$ with the diagram:
\begin{equation*}
\xymatrix{
X \ar[rr]^{\beta} \ar[rd]_{\alpha}&&  Z = \Alb(X/Y)\ar[ld]^{\gamma} \\
&Y = \Alb(X).   &
}
\end{equation*}
The fiber $G$ of $\gamma \colon Z \to Y$ is 
isomorphic to the Albanese torus $\Alb(F)$ of $F$, 
and thus the canonical bundle $K_G$ of $G$ is trivial. 
Furthermore, since the canonical bundle $K_Y$ of the base $Y$ 
is pseudo-effective, the theory of variation of Bergman kernels shows that $K_Z$ is pseudo-effective. Therefore, we conclude that $Z$ is a finite \'etale quotient of a torus. If $\dim Z > \dim Y$, this contradicts the definition of the Albanese map of $X$. This shows that $\dim Z = \dim Y$, implying that the irregularity $q(F)$ (and thus the augmented irregularity $\hat{q}(F)$) is zero.

Step (2) will be discussed in Section \ref{Sec-3}. To achieve Step (2), 
we focus on the notion of varieties of special type introduced by Campana. 
The structure theorem for projective manifolds shows that $X$ is rationally connected if it is projective. If $X$ is not projective, there exists a non-trivial holomorphic $2$-form $\eta$ on $X$. Then, by \cite[Theorem 1.4]{ZZ}, the tangent bundle $\TX{X}{}$ splits into $V \oplus W$ such that the tangent vectors in $V$ are truly flat. 
This result is crucial for our argument.
This shows that $V$ is a flat vector bundle, and thus, we can take 
the $\GL$-representation $\rho \colon \pi_1(X) \to \GL(r \mathord{:} \,\mathbb{C})$ associated with the flat bundle $V$. We show that $\Image \rho$ is finite, equivalently, 
the flat vector bundle $V$ is \'etale trivializable (i.e., there exists a finite \'etale cover $\nu \colon X' \to X$ such that $\nu^* V$ is trivial). This implies that some finite \'etale cover of $X$ admits a holomorphic $1$-form, contradicting our assumption $\hat{q}(X) = 0$. 
While it is generally difficult to directly prove that $\pi_1(X)$ is finite, it suffices for our purposes to show that $\Image \rho$ is finite. With this in mind, we prove that $\Image \rho$ is virtually abelian by showing that $X$ is a variety of special type (see Theorem \ref{thm-special}). Since the abelianization of $\pi_1(X)$ is finite due to $\hat{q}(X) = 0$, the universality of the abelianization guarantees that $\Image \rho$ is finite.

\medskip
We conclude with remarks on recent developments in holomorphic sectional curvatures. 
The approach used in this paper is somewhat algebraic-geometric. 
It would be of interest to explore a more differential-geometric approach, such as using holonomy groups or the de Rham decomposition, as suggested in \cite[Theorem 1.4, Remark 1.6]{ZZ}. 
Although the structure theorem presented here is powerful, several open problems remain in the study of holomorphic sectional curvatures. 
For instance, determining when a manifold admits (semi-)positive holomorphic sectional curvature remains challenging, even for rationally connected manifolds. For example, while the Hirzebruch surfaces are known to admit positive holomorphic sectional curvature \cite{Hit75} (see \cite{AHZ18, HC20}), it is not known whether its blow-up does (see \cite[Problem 3.8]{Mat20}). 
Holomorphic sectional curvatures continue to evolve in connection with other notions of curvature and positivity, though these aspects were not the focus of this paper. 
For further details, see \cite{Tan, CCZ, Ni19, Ni21, NZ22, Yan18b} and references therein.

Note that after we finished writing this paper, the result of \cite{BDPP13} was generalized to K\"ahler manifolds by \cite{Ou}. 
However, we believe that our methods, which focus on the notion of special varieties, provide a different perspective and are worth presenting.

\subsection*{Notation and Conventions}\label{subsec-notation}

We use the terms ``invertible sheaves'' and ``line bundles'' interchangeably and adopt the additive notation for tensor products (e.g.,\,$L+M:=L\otimes M$ for line bundles $L$ and $M$). Additionally, we use the terms ``locally free sheaves'' and ``vector bundles'' interchangeably. The term ``fibrations'' refers to a proper surjective morphism with connected fibers and the term ``finite \'etale covers'' refers to an unramified finite surjective morphism.

\subsection*{Acknowledgments}\label{subsec-ack}
The author expresses his sincere gratitude to Professor Xi Zhang for sharing the excellent preprint \cite{ZZ}, 
and to Dr.~Shiyu Zhang for giving a talk on \cite{ZZ} at the 30th International Conference on Finite or Infinite Dimensional Complex Analysis and Applications, which inspired the author to revisit the topic of this paper. 
The author was partially supported by Grant-in-Aid for Scientific Research (B) $\sharp$21H00976 from JSPS, 
Fostering Joint International Research (A) $\sharp$19KK0342 from JSPS, 
and JST FOREST Program, $\sharp$PMJFR2368 from JST.

\section{Varieties of special type}\label{Sec-2}

This section is devoted to the proof of Theorem \ref{thm-key}. To this end, we first show that any compact K\"ahler manifolds with semi-positive holomorphic sectional curvature belong to the class of \textit{varieties of special type}, as introduced by Campana. This result is proved by refining the argument in \cite{Mat20, Mat22}, and we provide the relevant details here for the reader's convenience.

\begin{theo}\label{thm-special}
Let $X$ be a compact K\"ahler manifold admitting a K\"ahler metric $g$ with semi-positive holomorphic sectional curvature. 
Then, the manifold $X$ is of special type in the sense of Campana.
\end{theo}

\begin{defi}[{\cite[Definitions 1.19, Proposition 1.25, Definitions 2.1, Theorem 2.22]{Cam04}}]\label{defi-special}
\ \\
(1) A compact K\"ahler manifold $X$ is said to be  \textit{of special type} in the sense of Campana 
if $X$ does not admit any almost holomorphic dominant rational map $\phi \colon X \dashrightarrow Y$ of general type. 

\smallskip
\noindent (2) A dominant rational map $\phi \colon X \dashrightarrow Y$ is said to be \textit{of general type} if 
some (equivalently, any) resolution  $\bar \phi \colon \Gamma \to Y$ of its indeterminacies 
satisfies that  
$$
\kappa \big(\mathcal{O}_{\Gamma}( {\bar \phi}^{*} K_{Y} )_{\sat}  \big) \geq m:=\dim Y >0, 
$$
where 
$$
\mathcal{O}_{\Gamma}( {\bar \phi}^{*} K_{Y} )_{\sat} \subset \Omega_{\Gamma}^{m}
$$
is the saturation of the invertible subsheaf 
defined by the pull-back of the canonical bundle $K_{Y}$ 
in the $m$-th exterior product $\Omega_{\Gamma}^{m}$ of the cotangent bundle $\Omega_{\Gamma}$. 
Here $\kappa(\bullet)$ denotes the Kodaira dimension. 
\end{defi}

\begin{proof}
Let $ \phi \colon X \dashrightarrow Y $ be an almost holomorphic dominant rational map, and let $ \bar{\phi} \colon \Gamma \to Y $ be a resolution of its indeterminacies with a bimeromorphic morphism $ \tau \colon \Gamma \to X $, as shown in the following diagram:
\begin{equation*}
\xymatrix{
\Gamma \ar[rr]^{\tau} \ar[rd]_{\bar{\phi}} && X \ar@{-->}[ld]^{\phi} \\
& Y. &
}
\end{equation*}
Set $ m := \dim Y > 0 $. In \cite[Theorem 3.1]{Mat22}, we showed that if the canonical bundle $ K_Y $ is pseudo-effective, then the invertible sheaf
\[
L := \big( \tau_* (\mathcal{O}_\Gamma(\bar{\phi}^* K_Y)) \big)^{**}
\]
defined by the reflexive hull (the double dual) is a Hermitian flat line bundle (equivalently, $c_{1}(L)=0$)
(see \cite[page 767]{Mat22}). This implies that the Kodaira dimension of $ \mathcal{O}_\Gamma(\bar{\phi}^* K_Y) $ cannot be greater than or equal to one. 

By essentially the same argument, we can show that if the saturation
\[
\mathcal{O}_\Gamma(\bar{\phi}^* K_Y)_{\sat} \subset \Omega_\Gamma^m
\]
of the pull-back $ \bar{\phi}^* K_Y $ 
in  the $ m $-th  exterior product $ \Omega_\Gamma^m $  is pseudo-effective, then the invertible sheaf
\[
M := \big( \tau_* (\mathcal{O}_\Gamma(\bar{\phi}^* K_Y)_{\sat}) \big)^{**}
\]
is also a Hermitian flat line bundle. 
Consequently, the Kodaira dimension of $ \mathcal{O}_\Gamma(\bar{\phi}^* K_Y)_{\sat} $ cannot be greater than or equal to one. 
Indeed, in proving the above conclusion for $L$, we used only the properties that $L$ is a pseudo-effective invertible subsheaf of $ \Omega^m_X $, 
and that $L$ coincides with the pull-back $ \phi^* K_Y $ on $X_1$, where $ X_1 $ is the preimage of a Zariski open set $ Y_1 \subset Y $ such that 
$$ \phi \colon X_1:=\phi^{-1}(Y_{1}) \to Y_1 $$ is a smooth fibration. These properties were used in \cite[Claim 3.5]{Mat22} via \cite[(3.1)]{Mat22}. Although \cite[(3.1)]{Mat22} is formulated on $ X_0 $ (the locus where $ \phi $ is holomorphic), it suffices if it holds on $ X_1 $. The invertible sheaf $M$ satisfies the first property if the saturation $ \mathcal{O}_\Gamma(\bar{\phi}^* K_Y)_{\sat} $ is pseudo-effective. 
Furthermore, we have $M = \phi^* K_Y$ on $X_1$, since
$$
\mathcal{O}_\Gamma(\bar{\phi}^* K_Y) \subset \Omega_\Gamma^m
$$
is a sub-line bundle over the smooth locus $Y_0$ of $\phi \colon X \to Y$ without taking saturation and 
$\tau \colon \Gamma \to X$ is an isomorphism over $Y_{0}$. 
Therefore, the same argument applied to $M$ shows that 
$
M = (\tau_* (\mathcal{O}_\Gamma(\bar{\phi}^* K_Y)_{\sat}))^{**}
$
is a Hermitian flat line bundle, just like $L$.

We will apply the conclusion mentioned above to show that $ X $ is of special type. By the definition of varieties of special type, we can take an almost holomorphic dominant rational map $ \phi \colon X \dashrightarrow Y $ and a bimeromorphic morphism $ \tau \colon \Gamma \to X $ admitting a fibration $ \bar{\phi} \colon \Gamma \to Y $ of general type. According to the definition of fibrations of general type, the Kodaira dimension $ \kappa(\mathcal{O}_{\Gamma}(\bar{\phi}^* K_{Y})_{\sat}) $ is equal to $ \dim Y > 0 $. 
In particular, the line bundle $ M $ is pseudo-effective. 
The conclusion mentioned above shows that $ M $ is a Hermitian flat line bundle. 
This contradicts the fact that $ \kappa(\mathcal{O}_{\Gamma}(\bar{\phi}^* K_{Y})_{\sat}) $ is greater than or equal to one. 
Therefore, the manifold $X$ does not admit any almost holomorphic dominant rational map $\phi \colon X \dashrightarrow Y$ of general type. 
\end{proof}

The fundamental group of a variety of special type is conjectured to be virtually abelian \cite[Conjecture 7.1]{Cam04}. This conjecture remains a challenging problem. Nevertheless, we can establish that the image of the $\GL$-representation of the fundamental group is virtually abelian. This weaker assertion, which is sufficient for our purpose, is proved in \cite{Cam04} with support from the profound results in \cite{Zuo99, Mok92, CCE15}. Furthermore, 
by the following corollary, we confirm that the Albanese map \( \phi \colon X \to Y := \Alb(X) \) satisfies the desired properties listed in Theorem~\ref{thm-main}, except for property $\rm{(b_{3})}$.

\begin{cor}\label{cor-special}
Let $(X, g)$ be a compact K\"ahler manifold admitting a K\"ahler metric $g$ with semi-positive holomorphic sectional curvature. 
Then, the following statements hold$:$
\begin{itemize}
\item[$\rm{(1)}$] Let $\rho \colon \pi_{1}(X) \to \GL(r \mathord{:} \,\mathbb{C})$ be 
a $\GL$-representation of the fundamental group $\pi_{1}(X)$. 
Then, the image $\Image \rho$ is a virtually abelian group. 
\item[$\rm{(2)}$]  The Albanese map $\phi\colon X \to Y:=\Alb(X)$ 
satisfies the properties listed in Theorem \ref{thm-main} except for property $\rm{(b_{3})}$
Furthermore, the fiber $F$ of $\phi\colon X \to Y=\Alb(X)$ admits a K\"ahler metric $g_{F}$ 
with semi-positive holomorphic sectional curvature. 
\end{itemize}
\end{cor}

\begin{proof}
Conclusion (1) is a direct consequence of \cite[Theorem 7.8]{Cam04} and Theorem \ref{thm-special}.

Conclusion (2) has been implicitly proved in \cite[Theorem 1.6]{Mat22}. For the reader's convenience, we will specify the relevant parts of the proof. The Albanese map $ X \to Y := \Alb(X) $ is a fibration by \cite[Theorem 7.4]{Cam04}. Thus, by \cite[Theorem 1.6]{Mat22}, there exists a holomorphic orthogonal splitting of the tangent bundle $T_X$:
\begin{align}\label{eq-split}
T_X \cong T_{X/Y} \oplus \phi^{*}T_Y.
\end{align}
The subbundle $ \phi^{*}T_Y \subset T_X $ is an integrable foliation (see \cite[page 772]{Mat22}), indicating that $ X \to \Alb(X) $ is a locally constant fibration (see, for example, \cite[Lemma 2.5]{Mat22}). Moreover, any tangent vectors in $ \phi^{*}T_Y \subset T_X $ are truly flat (see \cite[page 772]{Mat22}). Thus, by the holomorphic splitting and the Gauss-Codazzi type formula (see \cite[(2.2)]{Mat22}), the induced metric $ g_{T_{X/Y}} $ on $ T_{X/Y} $ satisfies
\begin{align*}
\big \langle \sqrt{-1}\Theta_g(e, \bar{e})(e), e \big \rangle_g &= 
\big \langle \sqrt{-1}\Theta_{g_{T_{X/Y}}}(e, \bar{e})(e), e \big \rangle_{g_{T_{X/Y}}}
\end{align*}
for any tangent vector $ e \in T_{X/Y} \subset T_X $. 
Here $\sqrt{-1}\Theta_h$ denotes the Chern curvature of a Hermitian vector bundle $(E, h)$:
$$
\sqrt{-1}\Theta_h := \sqrt{-1}\Theta_h(E) \in C^{\infty}(X, \Lambda^{1,1} \otimes \End(E)),
$$
and $\langle \bullet, \bullet \rangle_h$ is the Hermitian inner product with respect to $h$. See \cite[Section 2]{Mat22} for the precise notation and definition of truly flat tangent vectors. The left-hand side equals the holomorphic sectional curvature $H_g([f])$ evaluated at $ f $, which is semi-positive by assumption. This shows that the holomorphic sectional curvature of the induced metric $ g_F := g_{T_{X/Y}}|_F $ on $ T_F \cong T_{X/Y}|_F $ is semi-positive.
 \end{proof}

As explained in the proof of Theorem~\ref{thm-special}, the splitting \eqref{eq-split} is derived from the existence of a pseudo-effective invertible subsheaf $L \subset \Omega^{m}_{X}$ with the geometric property that $L = \phi^{*} K_{Y}$ holds on $X_{1}$. In fact, as shown by an excellent result in \cite{ZZ}, a splitting of the tangent bundle $T_X$ can be obtained solely from the existence of a pseudo-effective invertible subsheaf $L \subset \Omega^{m}_{X}$, without requiring the geometric condition. This result extends a generalization of Yau's conjecture to K\"ahler manifolds in the quasi-positive case (see \cite[Theorem 1.5]{ZZ}) and plays a crucial role in this paper.

\begin{theo}[{\cite[Theorem 1.4]{ZZ}}]\label{thm-ZZ}
Let $(X, g)$ be a compact K\"ahler manifold equipped with a K\"ahler metric $g$ with semi-positive holomorphic sectional curvature. 
Suppose that there exists a pseudo-effective invertible subsheaf $\mathcal{F} \subset \Omega_{X}^{m}$. 
Then, the tangent bundle $T_{X}$ admits a holomorphic orthogonal decomposition 
$$T_{X} \cong V \oplus W$$ 
such that $V$ is a nonzero flat vector bundle 
$($i.e.,\,there exists a $\GL$-representation $\rho \colon \pi_{1}(X) \to \GL(r \mathord{:} \,\mathbb{C})$ 
of the fundamental group $\pi_{1}(X)$ that determines $V$$)$.
\end{theo}
\begin{proof}
By \cite[Theorem 1.4]{ZZ}, there exists a holomorphic orthogonal decomposition 
$T_{X} \cong V \oplus W$ such that $V$ is a nonzero vector bundle and 
any tangent vectors in $V \subset T_{X}$ are truly flat. 
In the same way as in the proof of Corollary \ref{cor-special}, 
due to the holomorphic splitting, the induced metric $ g_{V} $ on $ V $ satisfies
\begin{align*}
\big \langle \sqrt{-1}\Theta_{g}(v, \bar{v})(f), f \big \rangle_g &= 
\big \langle \sqrt{-1}\Theta_{g_{V}}(v, \bar{v})(f), f \big \rangle_{g_{V}}
\end{align*}
for any tangent vector $ f \in V \subset T_{X} $ and any tangent vector $v \in T_{X}$. 
We can see that the left-hand side equals zero since $f \in V \subset T_{X}$ is a truly flat tangent vector. 
Hence, the induced metric on $V$ is Hermitian flat, and thus $V$ is a flat vector bundle.
\end{proof}

Based on the above preparations, we now proceed to prove the main theorem of this section.

\begin{theo}\label{thm-key}
Let $X$ be a compact K\"ahler manifold admitting a K\"ahler metric $g$ with semi-positive holomorphic sectional curvature. 
Consider the augmented irregularity 
$$
\hat{q}(X):=
\sup \{ 
q(X')\,|\, X' \to X \text{ is a finite \'etale cover} 
\} \in \mathbb{Z}_{\geq0} \cup \{\infty\},
$$
where 
$$q(X'):=\dim H^{0}(X', \Omega^{1}_{X'}) $$ is the irregularity of $X'$. 
Then, the following statements hold$:$
\begin{itemize}
\item[$\rm{(1)}$] The inequality $\hat{q}(X) \leq  \dim X$ holds. 
In particular, there exists a finite \'etale cover $X' \to X$ such that $q(X')=\hat{q}(X)$ holds. 
\item[$\rm{(2)}$]  If the augmented irregularity $\hat{q}(X)$ vanishes, 
then $X$ is a rationally connected projective manifold. 
\end{itemize}

\end{theo}
\begin{proof}
For any finite \'etale cover $\nu \colon X' \to X$, the K\"ahler metric $\nu^{*} g$ defined by the pull-back has semi-positive holomorphic sectional curvature. Hence, the Albanese map $\phi \colon X' \to Y' := \Alb(X')$ is a locally constant fibration (in particular, surjective) 
by Corollary \ref{cor-special}. This shows that 
$$
q(X') = \dim Y' \leq \dim X' = \dim X.
$$
This argument proves Conclusion (1).

We now consider Conclusion (2). Once we know that $X$ is a projective manifold, its rational connectedness is easily verified. Indeed, since Theorem \ref{thm-main} has been proved for projective manifolds, there exists an MRC fibration $X \to Y$ with the properties in Theorem \ref{thm-main}. Then, we can easily see that $\dim Y = \hat{q}(X)$. Hence, by the assumption $\hat{q}(X) = 0$, we conclude that $Y$ is a single point. This means that $X$ itself is the fiber of the MRC fibration (i.e.,\,it is rationally connected).

Suppose now that $X$ is not projective. Then, there exists a non-zero holomorphic $2$-form $\eta$ on $X$. Otherwise, the ample cone would coincide with the K\"ahler cone, which is non-empty, implying that $X$ is projective (see, for example, \cite[Proposition 3.3.2 and Corollary 5.3.3]{Huy05}).
We consider the reflexive hull 
$$
\mathcal{F} := (\langle \eta \rangle_{\mathcal{O}_{X}})^{**}
$$ 
of the sheaf $\langle \eta \rangle_{\mathcal{O}_{X}}$ generated by $\eta$. 
This sheaf $\mathcal{F} \subset \Omega_{\Gamma}^{2}$ is an effective invertible subsheaf 
$\mathcal{F} \subset \Omega_{X}^{2}$. Thus, by Theorem \ref{thm-ZZ}, the tangent bundle $T_{X}$ admits a holomorphic orthogonal decomposition $T_{X} \cong V \oplus W$ such that $V$ is a nonzero flat vector bundle. Let $\rho \colon \pi_{1}(X) \to \GL(r\mathord{:}\,\mathbb{C})$ be a $\GL$-representation that determines the flat vector bundle $V$, where $r := \rank V > 0$.

We will show that we may assume that $\Image \rho$ is an abelian group. By Corollary \ref{cor-special}, there exists an abelian normal subgroup $A \subset \Image \rho$ of finite index. 
We consider the \'etale cover $X' \to X$ corresponding to the kernel of the composition 
$$ \pi_{1}(X) \to \Image \rho \to \Image \rho / A.$$ 
The cover $\nu \colon X' \to X$ is a finite \'etale cover since $\Image \rho / A$ is a finite group. Note that $$0 \leq \hat{q}(X') \leq \hat{q}(X) = 0$$ holds by assumption. 
We now consider the induced $\GL$-representation 
$$
\rho' \colon \pi_{1}(X') \to \pi_{1}(X) \xrightarrow{\rho} \GL(r\mathord{:}\,\mathbb{C}).
$$ 
By construction, the image $\Image \rho'$ is contained in  $A$ (in particular, it is an abelian group). Furthermore, we still have the holomorphic orthogonal splitting 
\begin{align}\label{eq-pull}
T_{X'} = \nu^{*}T_{X} \cong \nu^{*}V \oplus \nu^{*}W
\end{align}
and the $\GL$-representation $\rho' \colon \pi_{1}(X') \to \GL(r\mathord{:}\,\mathbb{C})$ determines the flat subbundle 
$$\nu^{*}V \subset \nu^{*}T_{X} = T_{X'}.$$ 
Note that if $X'$ is projective, so is $X$.
Therefore, after replacing $X$ with the finite \'etale cover $X'$, 
we may assume that the image $\Image \rho$ is an abelian group.  

The first homology group $H_{1}(X, \mathbb{Z})$ can be regarded as the abelianization of $\pi_{1}(X)$. Since $\dim H^{1}(X, \mathbb{C}) = 0$ holds by the assumption $q(X) = 0$ and the K\"ahlerness of $X$, the first homology group $H_{1}(X, \mathbb{Z})$ is a torsion group (in particular, it is a finite group). Since $\Image \rho$ is an abelian group, the universality of abelianizations shows that $\rho \colon \pi_{1}(X) \to \GL(r\mathord{:}\,\mathbb{C})$ factors through 
$$\rho \colon \pi_{1}(X) \to H_{1}(X, \mathbb{Z}) \to \GL(r\mathord{:}\,\mathbb{C}).$$ 
This implies that $\Image \rho$ is a finite group, which means that $V$ is \'etale trivializable (i.e., there exists a finite \'etale cover $\nu \colon X' \to X$ such that the pull-back 
$\nu^{*}V $ is a trivial vector bundle on $X'$). 
Thus, by noting the splitting as in \eqref{eq-pull}, 
we can see that the cover $X'$ admits a non-zero $1$-form by 
$$\mathcal{O}_{X'}^{\oplus r} \cong (\nu^{*}V)^{*} \subset \Omega^{1}_{X'}.$$ 
This contradicts the assumption $\hat{q}(X) = 0$.
\end{proof}

\section{Proof of the main result}\label{Sec-3}

This section is devoted to the proof of Theorem \ref{thm-main}. To achieve this, we examine the Albanese map $X \to Y := \Alb(X)$ in detail and prove that the fiber $F$ of $X \to Y = \Alb(X)$ has vanishing augmented irregularity $\hat{q}(F) = 0$ 
to apply Theorem \ref{thm-key}. 
We begin with an elementary lemma, which is well-known to experts, but we provide proof for the reader's convenience.

\begin{lemm}\label{lem-fact}
Let $f \colon X \to Y$ be a smooth fibration between smooth varieties 
and $X \xrightarrow{g} Z \xrightarrow{\sigma} Y$ be its Stein factorization: 
\begin{equation*}
\xymatrix{
X \ar[rr]^{g} \ar[rd]_{f}&&  Z\ar[ld]^{\sigma} \\
&Y.   &
}
\end{equation*}
Then, the variety $Z$ is smooth, the morphism $\sigma \colon Z \to Y$ is a finite \'etale cover, 
and the morphism $g \colon X \to Z$ be a smooth fibration. 
\end{lemm}

\begin{proof}
The problem is local in $Y$, and thus, we may assume that $Y$ is an open ball. Let $z$ be a smooth point of $Z$, and take a point $x \in X$ such that $z = g(x)$. We consider the differential maps $df$, $dg$, and $d\sigma$ 
defined on the tangent bundles around $x$ and $z = g(x)$. 
Since $f \colon X \to Y$ is a smooth morphism, the differential map $df_{x} = d\sigma_{z} \circ dg_{x}$ is surjective. In particular, the differential map $d\sigma_{z}$ is surjective, which shows that $\sigma \colon Z \to Y$ is a finite \'etale cover 
on the smooth locus $Z_{\reg}$. This implies that the pushforward $\sigma_{*}\mathcal{O}_{Z}$ is a trivial vector bundle on $Y \setminus \sigma(Z_{\sing})$. Indeed, we consider the decomposition of the complement 
$Y \setminus \sigma(Z_{\sing})$ 
into connected components $V_{i}$:
$$
Y \setminus \sigma(Z_{\sing}) = \amalg_{i=1}^{d} V_{i}.
$$
We can observe that $\codim \sigma(Z_{\sing}) \geq 2$ holds since $Z$ is normal and $\sigma \colon Z \to Y$ is finite, and that $Y \setminus \sigma(Z_{\sing})$ is simply connected since $Y$ is smooth and simply connected. 
The morphism $\sigma|_{V_{i}} \colon V_{i} \to Y \setminus \sigma(Z_{\sing})$ is an \'etale cover with the connected $V_{i}$, and thus it must be an isomorphism. This shows that
$$
\sigma_{*}\mathcal{O}_{Z} |_{Y \setminus \sigma(Z_{\sing})} = \oplus_{i=1}^{d} \mathcal{O}_{V_{i}} 
\cong \oplus_{i=1}^{d} \mathcal{O}_{Y \setminus \sigma(Z_{\sing})}.
$$

On the other hand, by the flatness of $f \colon X \to Y$, we can see that 
$$
\sigma_{*}\mathcal{O}_{Z} = f_{*}g^{*}\mathcal{O}_{Z} = f_{*}\mathcal{O}_{X}
$$
is a reflexive sheaf. By the reflexivity of $\sigma_{*}\mathcal{O}_{Z}$ and the fact that $\codim \sigma(Z_{\sing}) \geq 2$, we can see that $\sigma_{*}\mathcal{O}_{Z}$ is the trivial locally free sheaf on $Y$. This implies that $Z$ is smooth and $\sigma \colon Z \to Y$ is a finite \'etale cover. We can also see that $g$ is smooth by $df_{x} = d\sigma_{z} \circ dg_{x}$.
\end{proof}

We now proceed to prove the main result of this paper.

\begin{proof}[Proof of Theorem \ref{thm-main}]
We first check that it is sufficient to prove both Conclusions (a) and (b) 
after replacing $X$ with its finite \'etale cover $\nu \colon X' \to X$. 
Indeed, this is obvious for Conclusion (a). 
Following the argument in \cite{CH19}, we check this for Conclusion (b). 
Suppose that there exists a finite \'etale cover $\nu \colon X' \to X$ 
such that $X'$ admits a locally constant fibration $\phi' \colon X' \to Y'$ with the properties listed in Theorem \ref{thm-main}. 
Let $\mathcal{F} \subset T_{X}$ be the (unique) saturated integrable subsheaf  
corresponding to an MRC fibration $X \dashrightarrow Y$ of $X$, 
so that a very general $\mathcal{F}$-leaf is a fiber of the MRC fibration $X \dashrightarrow Y$. 
The pull-back $\nu^{*}\mathcal{F} \subset T_{X'}$ corresponds to the locally constant MRC fibration $X' \to Y'$ of $X'$. 
This implies that $\nu^{*}\mathcal{F}$ is a locally free sheaf, and thus so is $\mathcal{F}$. 
Hence, by \cite[Corollary 2.11]{Hor07}, there exists a smooth RC fibration $\phi \colon X \to Y$ such that $T_{X/Y} = \mathcal{F}$.
Once we know that $K_{Y}$ is pseudo-effective, 
we can conclude that $X \to Y$ is the desired fibration by \cite[Theorem 1.6]{Mat22}.

To check that $K_{Y}$ is pseudo-effective, we consider the Stein factorization $X' \to Z \to Y$ of the composition $X' \xrightarrow{\nu} X \xrightarrow{\phi} Y$ (see the diagram below). 
Note that $Z$ is smooth and $Z \to Y$ is a finite \'etale cover by Lemma \ref{lem-fact}.
$$
\xymatrix{ & X  \ar[d]^{\phi} & X' \ar[d]^{\phi'}
\ar[l]_{\nu} \ar[d] \ar@/_4pc/[lld]_{g} \\ Z \ar[r]_{\sigma}& Y   & Y'.}
$$
We now check that $g \colon X' \to Z$ is an RC fibration (i.e.,\,a general fiber is rationally connected).
Let $z \in Z$ be a general point. 
The fiber $X_{y} := \phi^{-1}(y)$ at $y := \sigma(z)$ is rationally connected, since $y = \sigma(z)$ is also a general point in $Y$. 
By noting that $\sigma \colon Z \to Y$ is \'etale, 
we consider the decomposition into connected components: 
$$
\amalg_{w \in \sigma^{-1}(y)} X'_{w} = g^{-1}(\sigma^{-1}(y)) = \nu^{-1}(\phi^{-1}(y)) = \nu^{-1}(X_{y}), 
$$
where $X'_{w}:=g^{-1}(w)$ is the fiber of $g \colon X' \to Z$ at $w$. 
Since the fiber $X_{y} = \phi^{-1}(y)$ is rationally connected, so is the fiber $X'_{w} = g^{-1}(w)$. 
By noting that $z \in \sigma^{-1}(y)$, 
we can conclude that $X'_{z}$ is rationally connected 
(i.e.,\,$g \colon X' \to Z$ is an RC fibration). 

By the universality of MRC fibrations (see \cite{Cam92, KoMM92} for details), 
there exists a dominant rational map $Z \dashrightarrow Y'$. 
Since $Y'$ is a finite \'etale quotient of a torus (in particular, it has no rational curves), 
this map $Z \dashrightarrow Y'$ is actually an everywhere-defined fibration, 
and thus, it is a bimeromorphic morphism by $\dim Y' = \dim Z$. 
By comparing the canonical bundles, 
we can see that $K_{Z}$ is pseudo-effective 
since $Y'$ is smooth and $K_{Y'}$ is a trivial line bundle. 
This shows that $K_{Y}$ is also a pseudo-effective line bundle since $Z \to Y$ is an \'etale cover.
Thus, the fiber $X \to Y$ is the desired fibration by \cite[Theorem 1.6]{Mat22}.

\smallskip 

We will now prove Conclusion (b) after replacing $X$ with an appropriate finite \'etale cover. 
Let $\alpha \colon X \to Y := \Alb(X)$ be the Albanese map. 
By replacing $X$ with its finite \'etale cover, 
we may assume that 
$$
\hat{q}(X) = q(X) = \dim Y
$$
by Theorem \ref{thm-key}. 
Furthermore, as an additional reduction, 
we may assume that the fiber $F$ of the Albanese map $\alpha \colon X \to Y = \Alb(X)$ satisfies
$$
\hat{q}(F) = q(F)
$$
by replacing $X$ with its finite \'etale cover again. 
We will explain why this reduction is possible.
Note that all the fibers are isomorphic to each other since $\alpha \colon X \to Y = \Alb(X)$ is a locally constant fibration. 
The fiber $F$ also admits a K\"ahler metric with semi-positive holomorphic sectional curvature by Corollary \ref{cor-special}, and thus there exists a finite \'etale cover $F' \to F$ such that $q(F') = \hat{q}(F)$ by Theorem \ref{thm-key}. 
It is sufficient for this reduction to find a finite \'etale cover $F'' \to F'$ (whose irregularity still attains the augmented irregularity of $F$) 
such that it is induced by some finite \'etale cover $X' \to X$ of $X$.

By the definition of locally constant fibrations (see \cite[Definition 2.3]{MW}), 
the fiber product of $\alpha \colon X \to Y$ and the universal cover $\Unv{Y} \to Y$ is isomorphic to the product $\Unv{Y} \times F$. 
Furthermore, there exists a representation $\tau \colon \pi_{1}(Y) \to \Aut(F)$ such that 
\begin{align}\label{eq-isom}
X \cong (\Unv{Y} \times F) / \pi_{1}(Y),
\end{align}
where $\pi_{1}(Y)$ diagonally acts on the product $\Unv{Y} \times F$ via $\tau \colon \pi_{1}(Y) \to \Aut(F)$. 
We consider the induced homomorphism 
$$
\bar \tau \colon \pi_{1}(Y) \xrightarrow{\ \tau \ } \Aut(F) \xrightarrow{\quad} \Aut(\pi_{1}(F)) 
$$
and the core $H $ of $\pi_1(F')$ defined by 
$$
H := \bigcap \{ \bar \tau(\gamma) \cdot \pi_{1}(F') \mid \gamma \in \pi_{1}(Y) \} \subset \pi_1(F').
$$
By construction, the core $H$ is invariant with respect to the action of $\bar \tau$. 
Furthermore, the core $H$ is a subgroup of $\pi_{1}(F)$ of finite index. 
Indeed, let $N$ be the index of $\pi_{1}(F') \subset \pi_{1}(F)$, which is finite. 
Then, for each $\gamma \in \pi_{1}(Y)$, we can easily see that 
$$\big( \bar \tau(\gamma) \cdot \pi_{1}(F') \big) \cap \pi_{1}(F')
$$ is a subgroup of $\pi_{1}(F')$ of index at most $N$. 
Furthermore, the number of subgroups of a fixed finite index is finite, 
and thus $H$ is actually the intersection of a finite number of $\bar \tau(\gamma) \cdot \pi_{1}(F')$. 
This implies that $H$ is a subgroup of $\pi_{1}(F)$ of finite index.

Consider the finite \'etale cover $F'' \to F$ corresponding to the $\bar \tau$-invariant subgroup $H \subset \pi_{1}(F)$ of finite index. 
By construction, since $F'' \to F$ factors through $F'' \to F' \to F$, 
we have 
$$
q(F'') \geq q(F') = \hat{q}(F), 
$$ 
and thus $q(F'' )$ attains the augmented irregularity $\hat{q}(F)$ of $F$.
Furthermore, since $$\pi_{1}(F'') = H \subset \pi_{1}(F)$$ is $\bar \tau$-invariant, 
we obtain the representation $\tau'' \colon \pi_{1}(Y) \to \Aut(F'')$ 
by noting that $\Aut(F)$ and $\Aut(F'')$ can be regarded as the subgroups of $\Aut(\Unv{F})$ that fix $\pi_{1}(Y)$ and $H$, respectively. 
As in \eqref{eq-isom}, we define $X''$ by the quotient of $\Unv{Y} \times F''$ by the diagonal action induced by the representation $\tau''$. 
Then, by construction, the manifold $X''$ admits 
the induced locally constant fibration with fiber $F''$: 
\begin{align*}
X'' := (\Unv{Y} \times F'') / \pi_{1}(Y) \to \Unv{Y}/\pi_{1}(Y) \cong Y. 
\end{align*}

\smallskip

After replacing $X$ with the finite \'etale cover constructed above, 
we consider the Albanese map $\alpha \colon X \to Y = \Alb(X)$ with the fiber $F$. 
We will prove that the irregularity $q(F)$ (and thus, the augmented irregularity $\hat{q}(F)$) is zero, assuming that 
$$\text{
$\hat{q}(F) = q(F)$ \quad and \quad  $\hat{q}(X) = q(X) = \dim Y$. 
}
$$
To this end, we consider the relative Albanese map $\beta \colon X \to Z := \Alb(X/Y)$ 
associated with the Albanese map $\alpha \colon X \to Y := \Alb(X)$: 
\begin{equation*}
\xymatrix{
X \ar[rr]^{\beta} \ar[rd]_{\alpha}&&  Z = \Alb(X/Y)\ar[ld]^{\gamma} \\
&Y = \Alb(X).   &
}
\end{equation*}
Since $\alpha \colon X \to Z$ is locally trivial, so is $\gamma \colon Z \to Y$. 
For any point $y \in Y$, 
the fiber $G_{y} := \gamma^{-1}(y)$ is the Albanese torus $\Alb(F_{y})$, 
where $F_{y} := \alpha^{-1}(y)$ is the fiber of $\alpha \colon X \to Y$ at $y$. 
In particular, the canonical bundle $K_{G_{y}}$ is a trivial line bundle. 
Furthermore, since $Y$ is a finite \'etale quotient of a torus, the canonical bundle $K_{Y}$ is a Hermitian flat line bundle. 
Thus, the direct image sheaf
$$
\gamma_{*}(\mathcal{O}_{Z}(K_{Z/Y} + \gamma^{*}K_{Y})) = \gamma_{*}(\mathcal{O}_{Z}(K_{Z}))
$$
is a line bundle admitting a smooth Hermitian metric with semi-positive curvature 
by the positivity of direct images \cite[Theorem 2.6]{Wan21} 
(see also \cite[Theorem 1]{PT18} and \cite{BP08, HPS18}). 
Moreover, since $K_{G_{y}}$ is a trivial line bundle, 
the evaluation morphism 
$$
\gamma^{*}\gamma_{*}(\mathcal{O}_{Z}(K_{Z})) \to \mathcal{O}_{Z}(K_{Z})
$$
is surjective, which shows that $\mathcal{O}_{Z}(K_{Z})$ admits a smooth Hermitian metric with semi-positive curvature 
(in particular, it is pseudo-effective). 
Thus, by \cite[Theorem 1.6]{Mat22}, 
there exists a finite \'etale cover $T \to Z$ from a torus. 
Consider the fiber product $X' := X \times_Z T$ with the following diagram:
\[
\xymatrix{
X' := X \times_Z T \ar[r] \ar[d] & X  \ar[d]\\
T \ar[r] & Z.  \\
}
\]
Note that $X' \to X$ is a finite \'etale cover by construction. 
If $\dim Z > \dim Y$, then we obtain 
$$
\hat{q}(X) = q(X') \geq \dim Z > \dim Y = q(X) = \hat{q}(X),
$$
which is a contradiction. 
Hence, we have $\dim Y = \dim Z$, and thus $\gamma \colon Z \to Y$ is bimeromorphic. The bimeromorphic morphism $\gamma \colon Z \to Y$ is an isomorphism since both $K_{Y}$ and $K_{Z}$ are Hermitian flat line bundles.

$Z$ is actually isomorphic to $Y$. 
Consequently, we can obtain 
$$
\hat{q}(F) = q(F) = 0
$$
from the definition of relative Albanese maps and the choice of $F$. 

Theorem \ref{thm-key} shows that $F$ is a rationally connected projective manifold by $\hat{q}(F) = 0$, 
which completes the proof of Conclusion (b). 
Since the rationally connected fiber $F$ is simply connected, 
we can see that 
$$\pi_{1}(X) \cong \pi_{1}(Y) \cong \mathbb{Z}^{\oplus 2 \dim Y}$$
by the homotopy exact sequence of the fibration $X \to Y$. 
Thus, we can see that $\pi_{1}(X)$ is an abelian group, finishing the proof of Conclusion (a).
\end{proof}

At the end of this paper, we  prove Corollary \ref{cor-Yau} 
by applying Theorem \ref{thm-main} and Corollary \ref{cor-special}. 

\begin{proof}[Proof of Corollary \ref{cor-Yau}]
By Theorem \ref{thm-main}, there exists an MRC fibration $X \to Y$ such that $Y$ is a finite \'etale quotient of a torus. Then, by the proof of Corollary \ref{cor-special}, the tangent bundle $T_X$ splits as in \eqref{eq-split}, and any tangent vectors in $\phi^{*}T_Y \subset T_X$ are truly flat. This proves the desired inequality.
\end{proof}


\end{document}